\documentclass[english,letterpaper,11pt,reqno]{amsart}

\usepackage{appendix}
\usepackage{amsbsy}
\usepackage{amsfonts}
\usepackage{amsmath}
\usepackage{amssymb}
\usepackage{amsthm}
\usepackage{graphicx}
\usepackage{ifthen}
\usepackage{textcomp}
\usepackage{soul}
\usepackage{enumitem,kantlipsum}
\usepackage{tikz}
\usepackage{mathrsfs}
\usepackage{cancel}
\usepackage{lmodern}
\usepackage{dsfont}
\usepackage[bookmarksnumbered,colorlinks]{hyperref}
\hypersetup{colorlinks=true, linkcolor=blue, citecolor=blue}
\newcommand{\lra}{\longrightarrow}

\newcommand{\RR}{\mathbb{R}}

\newcommand{\uu}{\mathscr{U}}
\newcommand{\cc}{\mathcal{C}}
\newcommand{\Lev}{\mathcal{S}}
\newcommand{\sff}{\mathrm{II}}

\newtheorem{thm}{Theorem}

\newtheorem{lemma}{Lemma}

\newcommand{\beqa}{\begin{eqnarray}}
\newcommand{\beq}{\begin{equation}}
\newcommand{\eeqa}{\end{eqnarray}}
\newcommand{\eeq}{\end{equation}}

\newcommand\imp{\hspace{.2in}\Rightarrow\hspace{.2in}}

\newcommand\cd[2]{\nabla_{\!#1}{#2}}
\newcommand\conn[2]{\overline{\nabla}_{\!#1}{[#2]}}

\newcommand\nf{\nabla\!f}

\newcommand\gnf{g_{\scalebox{0.5}{\emph{$\nf$}}}}
\newcommand\gN{g_{\scalebox{0.5}{\emph{N}}}}
\newcommand\hc{h_{\scalebox{.5}{\emph{b,c}}}}
\newcommand\gpw{h_{\scriptscriptstyle \Omega}}

\newcommand\comma{\hspace{.2in},\hspace{.2in}}

\newcommand*{\defeq}{\mathrel{\vcenter{\baselineskip0.5ex \lineskiplimit0pt
                     \hbox{\scriptsize.}\hbox{\scriptsize.}}}
                     =}

\usepackage{enumitem, color, amssymb}

\begin{document}
\title[]{Finsler pp-Waves and the Penrose Limit}
\author{Amir Babak Aazami, Miguel \'{A}ngel Javaloyes, Marcus C. Werner}
\address{Amir Babak Aazami\hfill\break\indent Department of Mathematics \hfill\break\indent Clark University\hfill\break\indent Worcester, MA 01610, USA}
\email{Aaazami@clarku.edu}
\address{Miguel \'{A}ngel Javaloyes \hfill\break\indent Department of Mathematics\hfill\break\indent  University of Murcia \hfill\break\indent 30100 Murcia, Spain}
\email{majava@um.es}
\address{Marcus C. Werner \hfill\break\indent Zu Chongzhi Center for Mathematics and Computational Sciences \hfill\break\indent Duke Kunshan University \hfill\break\indent Kunshan, Jiangsu 215316, China}
\email{marcus.werner@dukekunshan.edu.cn}
\maketitle

\begin{abstract}
The Penrose plane wave limit is a remarkable property of Lorentzian spacetimes. Here, we discuss its extension to Finsler spacetimes by introducing suitable lightlike coordinates and adapting the Lorentzian definition of pp-waves. New examples of such Finsler pp-waves are also presented.
\end{abstract}

\section{Introduction}
Among R. Penrose's many accomplishments in general relativity are two foundational results on \emph{plane waves}, a distinguished class of spacetimes  modeling radiation propagating at the speed of light.  These metrics have a long and rich history within the field of general relativity; they sit inside the more general family of \emph{pp-wave} spacetimes introduced by J. Ehlers and W. Kundt \cite{ehlerskundt}, which themselves comprise an important subclass of the family of so called \emph{Brinkmann spacetimes} due to H. W. Brinkmann \cite{brinkmann}, which are spacetimes containing a parallel (covariantly constant) lightlike vector field.  A very comprehensive recent survey of plane waves can be found in \cite{AMS}; pp-waves are now also studied purely in a mathematical context (see, e.g., \cite{flores,globke,leistner,FS-Ehlers}), in addition to their continued usage in gravitational physics (see, e.g., \cite{blau,blau2}).  Of their many properties both physical and mathematical, two of the most noteworthy were discovered by Penrose himself.  The first of these, \cite{remarkable}, is that plane waves are never globally hyperbolic: this is ``remarkable", to borrow Penrose's own description, all the more so since plane waves are known to be geodesically complete.  The second result, \cite{penPW}, no less remarkable, is that every spacetime has a plane wave in a certain well defined limit, a local construction now known as Penrose's ``plane wave limit," perhaps the most important realization of a more general notion of ``spacetime limit" due to R. Geroch \cite{gerochL}.  
\newline
\indent Given the distinguished position that plane waves occupy\,---\,and especially in light of Penrose's result that every Lorentzian metric admits one as a limit\,---\,it is worthwhile to ask whether there exist analogues of them in \emph{other} geometries of indefinite signature, and, if so, whether Penrose's plane wave limit carries over to such settings as well. A natural direction in which to take this question is that of Finsler geometry of Lorentzian signature, the setting of so called \emph{Finsler spacetimes}, especially in light of their many recent physical applications; see, e.g., \cite{BJS20,HPV20,JS20,Kos1,Kos2,PfeiWol,Perlick} and the references therein.  Indeed, there are already examples of ``Finsler pp-waves" in the literature (see, e.g., \cite{fusterpabst}).  In pursuing our question, we are less motivated by the connection between pp-waves, gravitational radiation, and Einstein's equation\,---\,indeed, there is no agreed upon analogue of the latter in Finsler geometry\,---\,and more by the fact that, via Penrose's limit, plane waves play a central role in Lorentzian geometry per se.  What role may they play in Finsler geometry?
\newline
\indent
In this paper we attempt an answer to this question by first introducing a general definition of Finsler pp-wave, a definition that subsumes the isolated examples of Finsler pp-waves already in the literature; indeed, in Section \ref{sec:Examples} below we introduce additional examples of Finsler pp-waves in accord with our definition.  Second, we show that there does exist a notion of ``plane wave limit" in the Finslerian setting.  These two facts, we argue, make them worthy of study as Finslerian objects.  Throughout our paper, at every step of our construction, we carefully present the modifications required to pass from Lorentzian to Finslerian geometry: this includes presenting the Finsler analogue of lightlike coordinates (see, e.g., \cite{pen}), the Finsler analogue of the invariant definition of pp-wave in terms of the Riemann curvature tensor (see, e.g., \cite{globke}), and, finally, the Finsler analogue of the construction of Penrose's plane wave limit itself.

\section{Preliminaries}
\label{sec:Preliminaries}
Let $M$ be a (connected) manifold, $TM$ its tangent bundle, and $\pi\colon TM\rightarrow M$ the natural projection. Let us consider a  connected open subset $A\subset TM\setminus \bf 0$ which is conic, namely, satisfying the property that $\lambda v\in A$ for all $v\in A$ and $\lambda>0$. Further assume that $A$ has smooth boundary, that each $A_p=A\cap T_pM$ is nonempty for all $p\in M$, and denote  $\bar A$ its closure in $TM\setminus \mathbf{0}$. Given a function $L\colon \bar A \subset TM \rightarrow [0,\infty)$, we will say that   $(M,L)$ is a {\em Finsler spacetime} if $L$ is a
smooth function which is positive homogeneous of degree $2$ when restricted to each $\bar A_p \defeq T_pM\cap \bar A$, its fundamental tensor
\begin{equation}\label{fundten}
g_{v}(u,w)=\frac{1}{2}\left.\frac{\partial^2}{\partial s\partial t} L(v+tu+sw)\right|_{t=s=0}
\end{equation}
for $v\in \bar A$ and $u,w\in T_{\pi(v)}M$   has signature $(+-\dots-)$ and  the  boundary of $\bar A$  in $TM\setminus \{\mathbf{0}\}$ coincides with $\cc\defeq L^{-1}(0)$, which is called the lightcone of $L$. Observe that in this case, $\bar A_p$ is convex and salient for every $p\in M$ and the indicatrix $\Sigma=L^{-1}(1)$ is a strongly convex hypersurface when restricted to each tangent space, namely, each $\Sigma_p=\Sigma\cap T_pM$ is strongly convex.

Observe that there are quite a few subtleties in the definition of a Finsler spacetime, but we will adopt the one firstly considered in \cite{JS14}. Amongst the subtetlies to bear in mind we can point out that
\begin{itemize}[leftmargin=*]
\item in some definitions, beginning with that of Beem \cite{Beem71}, $L$ is defined in the whole $TM$. Observe that as explained in \cite{BJS20}, for an observer $v\in \Sigma$ one can consider $g_v$ as a positive definite metric in its restspace, without any need of considering $L$ defined away from the causal cone.
\item In others, some non-smooth directions are allowed \cite{Perlick,AJ16,CaStan,Minguzzi} or they must be smooth up to some power \cite{PfeiWol} (relevant in the lightlike directions),
\item there are others that do not consider the lightlike directions as a part of the model \cite{Asanov},
\item there are some models which appear in other contexts with  slightly different properties \cite{Kos1,Kos2}.
\end{itemize}
Associated with the Lorentz-Finsler metric $L$, there is another anisotropic tensor, usually called the Cartan tensor, defined as 
\begin{equation}\label{Cartan}
C_{v}(u,w,z)=\frac{1}{4}\left.\frac{\partial^3}{\partial r\partial s\partial t} L(v+tu+sw+rz)\right|_{t=s=r=0},
\end{equation}
for any $v\in \bar A$ and $u,w,z\in T_{\pi(v)}M$.
This symmetric anisotropic tensor is what makes different Finsler spacetimes from the classical Lorentzian geometry. It is straightforward to check that 
\begin{equation}\label{homoCartan}
C_{v}(v,u,w)=C_{v}(u,v,w)=C_{v}(u,w,v)=0,
\end{equation}
for any $v\in \bar A$ and $u,w\in T_{\pi(v)}M$. Moreover, the Levi-Civita--Chern anisotropic connection is a very useful tool for the study of Finsler spacetimes (see \cite{Jav19,Jav20,JSV22} for more details about anisotropic connections and calculus). Recall that an anistropic connection can be thought as a connection which depends on directions. This means that for every $v\in \bar A$ and $X,Y\mathfrak{X}(M)$, one obtains a different value $\nabla^{v}_XY\in T_{\pi(v)}M$. In particular, given a chart $(\uu,\varphi)$, the Christoffel symbols are functions $\Gamma^k_{ij}\colon \bar A\cap T\uu\rightarrow \RR$ which are homogeneous of degree zero and $\nabla^{v}_{\partial_i}\partial_j=\Gamma^k_{ij}(v)\partial_k$, where $\partial_1,\ldots,\partial_n$ are the partial vector fields of the chart. If we fix a vector field $V\in{\mathfrak{X}}(\uu)$ which is \emph{$\bar A$-admissible}, that is to say, taking values in $\bar A$, we obtain then an affine connection $\nabla^{\scalebox{0.5}{\emph{V}}}$ in $\uu$ with Christoffel symbols $\Gamma^i_{jk}\circ V$. The Levi-Civita--Chern connection can be characterized in terms of the associated affine connections $\nabla^{\scalebox{0.5}{\emph{V}}}$. Indeed, it is the only one such for which
\begin{enumerate}[leftmargin=*]
 \item[1.] $\nabla^{\scalebox{0.5}{\emph{V}}}$ is {\it torsion-free}, namely,
$\nabla^{\scalebox{0.5}{\emph{V}}}_XY-\nabla^{\scalebox{0.5}{\emph{V}}}_YX=[X,Y]$
for all $X, Y\in{\mathfrak X}(\uu)$,
\item[2.] $\nabla^{\scalebox{0.5}{\emph{V}}}$ is {\it almost $g$-compatible}, namely
\[X( g_{\scalebox{0.5}{\emph{V}}}(Y,Z))=g_{\scalebox{0.5}{\emph{V}}}(\nabla^{\scalebox{0.5}{\emph{V}}}_XY,Z)+g_{\scalebox{0.5}{\emph{V}}}(Y,\nabla^{\scalebox{0.5}{\emph{V}}}_XZ)+2 C_{\scalebox{0.5}{\emph{V}}}(\nabla^{\scalebox{0.5}{\emph{V}}}_XV,Y,Z),\]
where $X,Y,Z\in{\mathfrak X}(\uu)$ and $g_{\scalebox{0.5}{\emph{V}}}$ and $C_{\scalebox{0.5}{\emph{V}}}$ are the classical tensors obtained when \eqref{fundten} and \eqref{Cartan} are evaluated in the vector field $V$.
\end{enumerate}
Observe that almost $g$-compatibility is equivalent to $\nabla g=0$ when this tensor derivative is computed using the anisotropic calculus developed in 
\cite{Jav19,Jav20}. Moreover, there is also a Koszul formula that determines $\nabla^{\scalebox{0.5}{\emph{V}}}$:
\begin{multline}
 2 g_{\scalebox{0.5}{\emph{V}}}(\nabla^{\scalebox{0.5}{\emph{V}}}_XY,Z)= X (g_{\scalebox{0.5}{\emph{V}}}(Y,Z))-Z (g_{\scalebox{0.5}{\emph{V}}}(X,Y))+Y (g_{\scalebox{0.5}{\emph{V}}}(Z,X)\\
+g_{\scalebox{0.5}{\emph{V}}}([X,Y],Z)+g_{\scalebox{0.5}{\emph{V}}}([Z,X],Y)-g_{\scalebox{0.5}{\emph{V}}}([Y,Z],X)\\
2\big(\!-C_{\scalebox{0.5}{\emph{V}}}(\nabla^{\scalebox{0.5}{\emph{V}}}_XV,Y,Z)-C_{\scalebox{0.5}{\emph{V}}}(\nabla^{\scalebox{0.5}{\emph{V}}}_YV,Z,X)+C_{\scalebox{0.5}{\emph{V}}}(\nabla^{\scalebox{0.5}{\emph{V}}}_ZV,X,Y)\big).\label{koszul}
\end{multline}
Observe that when $V$ is parallel, namely, $\nabla^{\scalebox{0.5}{\emph{V}}}_XV=0$ for all $X\in \mathfrak{X}(\uu)$, then the above Koszul formula coincides with the Koszul formula for $g_{\scalebox{0.5}{\emph{V}}}$. This means that when $V$ is parallel, the Levi-Civita--Chern connection $\nabla^{\scalebox{0.5}{\emph{V}}}$ of $L$ coincides with the Levi-Civita connection of $g_{\scalebox{0.5}{\emph{V}}}$. Even if it is not always possible to choose a parallel extension of any vector $v\in \bar A$, one can find a pointwise parallel vector field, namely, if $p=\pi(v)$, there exists an extension  $V\in \mathfrak{X}(\uu)$ for a certain neighborhood $\uu\subset M$ of $p$ such that $(\nabla^{\scalebox{0.5}{\emph{V}}}_XV)_p=0$ for all $X\in \mathfrak{X}(\uu)$ (see \cite[Prop. 2.13]{Jav20}). Indeed, with these extensions we can compute the Chern curvature tensor of $L$ very easily, If $R_v(X,Y)Z$ is the value of the Chern curvature tensor (which is an anisotropic tensor) for $v\in \bar A$ and $X,Y,Z\in \mathfrak{X}(M)$, let us choose a pointwise parallel extension $V$ of $v$, then
\[R_v(X,Y)Z=R^{\scalebox{0.5}{\emph{V}}}_p(X,Y)Z,\]
where $R^{\scalebox{0.5}{\emph{V}}}$ is the curvature tensor of $\nabla^{\scalebox{0.5}{\emph{V}}}$. This is because in the general expression of the Chern curvature tensor computed using $\nabla^{\scalebox{0.5}{\emph{V}}}$, apart from $R^{\scalebox{0.5}{\emph{V}}}$ there are some additional tensorial terms evaluated in $(\nabla^{\scalebox{0.5}{\emph{V}}}_XV)_p$ (see \cite[Prop. 2.5]{Jav20}). In particular, when $V$ is parallel, the curvature tensor of the Levi-Civita connection of $g_{\scalebox{0.5}{\emph{V}}}$ coincides with the Chern curvature tensor $R_{\scalebox{0.5}{\emph{V}}}$ evaluated at $V$.

A function $f\colon M\rightarrow \RR$ admits a gradient with respect to $L$, denoted $\nabla\!f$, if there exists a vector field metrically equivalent to $df$, namely, which satisfies
\[\gnf(\nf,X)=df(X)\]
for all $X\in{\mathfrak X}(M)$.
 Moreover, in this case the Hessian of $f$ is defined as  the anisotropic tensor $H^f(X,Y)=\cd{X}{d f}(Y)$, where $\nabla$ is the Levi-Civita--Chern connection of $L$. Observe that there is a dependence on $v\in A$ in the sense that
\[H^f_v(X,Y)=(\cd{X}{df})_v(Y)= X(df(Y))-df(\nabla^{\scalebox{0.5}{\emph{V}}}_XY).\]
Let us observe that $H^f$ is symmetric as 
\[H^f_v(X,Y)= X(df(Y))-df(\nabla^{\scalebox{0.5}{\emph{V}}}_XY)=X(Y(f))(\pi(v))-\nabla^{\scalebox{0.5}{\emph{V}}}_XY (f),\]
and $\nabla^{\scalebox{0.5}{\emph{V}}}_XY-\nabla^{\scalebox{0.5}{\emph{V}}}_YX=[X,Y]$ (the Levi-Civita--Chern connection is torsion-free), and
\begin{equation}
H^f_{{\scalebox{0.5}{\emph{$\nf$}}}}(X,Y)= \gnf(\nabla^{\nf}_{\!X}(\nf),Y).
\end{equation}
This follows from the above formula, because using the almost-compatibility of $\nabla $ with $g$,
\begin{align*}
\gnf(\nabla^{\nf}_X (\nf),Y)=&X(\gnf(\nf,Y))-\gnf(\nf, \nabla^{\nf}_XY)\\
=&X(Y(f))-\nabla^{\nf}_XY (f)=H^f_{\nf}(X,Y),
\end{align*}
(the Cartan term vanishes because homogeneity \eqref{homoCartan}).
\begin{lemma}
 A smooth function $f\colon M\rightarrow \RR$ admits a gradient with respect to a Finsler spacetime $(M,L)$ if and only if  $df|_A>0$. Moreover, in this case the gradient is unique.
 \end{lemma}
 \begin{proof}
 For the implication to the right, observe that if $\nf$ is lightlike at $p\in M$, this means that the kernel of $df_p$ is tangent to $\cc_p$. As $A_p$ is convex, it remains on one side of the hyperplane $\ker (df_p)$, and as $\gnf$ is a Lorentzian-type metric (with index $n-1$) it follows that $df_p|_A>0$, because the lightlike cone of $\gnf$ remains on the same side of $\ker(df_p)$. If $\nf$ is timelike at $p\in M$, then $\ker (df_p)$ does not touch $\bar A_p$ and as $df_p(\nf)=\gnf(\nf,\nf)=L(\nf)>0$, it follows that $df_p|_A>0$.
 
 For the implication to the left, we know that $\ker(df)\cap A=\emptyset$. If $\ker(df_p)$ is tangent to $\cc_p$, then there exists $v\in \cc_p\cap \ker(df)$. Then it is easy to see that 
 \begin{equation}\label{df}
 df_p=\lambda g_{v}(v,\cdot)=g_{\lambda v}(\lambda v,\cdot)
 \end{equation}
  and therefore $\nf=\lambda v$. If $\ker(df_p)$ is not tangent to $\cc_p$, then there exists $v\in \Sigma_p$ where the minimum distance between $\ker(df_p)$ and $\Sigma_p$ is attained. Then \eqref{df} holds for some $\lambda$. This also shows that the gradient is unique, because the strict convexity of $\Sigma_p$ implies that the distance is attained in a unique point.
 \end{proof}
\begin{lemma}\label{geoflow}
If a function $f\colon M\rightarrow\RR$ has a gradient field with constant $L$-norm, then its flow is given by geodesics.
\end{lemma}
\begin{proof}
As the gradient field has constant $L$-norm, it follows that $X(L(\nf))=0$ and then
\[0=2\gnf(\nabla^{\nf}_X \nf,\nf)=2H^f_{\scalebox{0.5}{\emph{$\nf$}}}(X,\nf)=2H^f_{\scalebox{0.5}{\emph{$\nf$}}}(\nf,X)=2\gnf(\nabla^{\nf}_{\nf} \nf,X),\]
and as this holds for all $X\in {\mathfrak X}(M)$, it follows that $\nabla^{\scalebox{0.5}{\emph{$\nf$}}}_{\!\nf}\nf=0$ and $\nf$ is geodesic.
\end{proof}
\section{Lightlike coordinates and their properties}
\label{sec-lightlike}
\label{theorem:nc}
Although lightlike coordinates exist in all dimensions, we present them here in dimension 4, for convenience. In the following, when a chart $(\uu,\varphi)$ is fixed, we will denote by $g_{ij}\colon\bar A\cap T\uu\rightarrow \RR$ the coordinates of the fundamental tensor $g$ in \eqref{fundten}, namely, $g_{ij}(v)=g_{v}(\partial_i,\partial_j)$.

\begin{lemma}[Lightlike coordinates]
Let $(M,L)$ be a Finsler spacetime and $f$ a smooth function defined on an open subset $\mathscr{U} \subseteq M$.  If $N \defeq \nf$ is a lightlike vector field, then there exist coordinates about any point in $\mathscr{U}$ in which the metric $g_{\scalebox{0.5}{N}}$ has components
\beqa
\label{nc}
(g_{ij}(N)) = 
    \left(
      \begin{array}{cccc}
        0 & 1 & 0 & 0\\
        1 &  g_{11}(N) & g_{12}(N) & g_{13}(N)\\
        0 & g_{21}(N) & h_{22} & h_{23}\\
        0 & g_{31}(N) & h_{32} & h_{33}
      \end{array}
    \right),
\eeqa
with $(h_{ij})$ a positive-definite $2 \times 2$ matrix.  
\label{lem-lightlike}
\end{lemma}

\begin{proof}
Of course, ``lightlike" means that $N$ is nowhere vanishing yet satisfies $L(N)= 0$; the former implies that the level sets $\Lev_c \defeq f^{-1}(c)$ are embedded hypersurfaces, while the latter implies that $N$ has geodesic flow, $\nabla^{\scalebox{0.5}{\emph{N}}}_N{N} = 0$, where $\nabla^{\scalebox{0.5}{\emph{N}}}$ is the Levi-Civita--Chern connection of $L$ (recall Lemma \ref{geoflow}).  At any $p \in \Lev_c$, $g_{\scalebox{0.5}{$N_p$}}(N_p,N_p)  =  0$ implies that the induced metric by $g_{\scalebox{0.5}{$N_p$}}$ on $\Lev_c$ is degenerate.  Let $(x^0,x^1,x^2,x^3)$ denote coordinates within $\uu$ in which  $N  = \partial_0$.  Then by the almost compatibility of $g$ and $\nabla$,
$$
N\big(\gN(N,{\partial_i}\big) = \gN{\underbrace{\nabla^{\scalebox{0.5}{\emph{N}}}_{\!N}N}_{0}}{\partial_i} + \gN(N,N){\underbrace{\nabla^{\scalebox{0.5}{\emph{N}}}_{\!N}\partial_i}_{\cd{\partial_i}{N}}\,} = 0,
$$
so that each $c_i\defeq g_{0i}(N)$, $i=0,\ldots,3$, is independent of $x^0$.  Of course, at least one of these $c_i$'s must be \emph{nonzero}, otherwise each $N_p^{\perp}$ would be four-dimensional.  Let us assume that $c_1 \neq 0$.  Thus at the moment our metric $\gN$ in the coordinates $(x^0,x^1,x^2,x^3)$ takes the form
\beqa
(g_{ij}(N)) = 
    \left(
      \begin{array}{cccc}
        0 & c_1 & c_2 & c_3\\
        c_1 &  g_{11}(N) & g_{12}(N) & g_{13}(N)\\
        c_2 & g_{21}(N) & g_{22}(N) & g_{23}(N)\\
        c_3 & g_{31}(N) & g_{32}(N) & g_{33}(N)
      \end{array}
    \right) \comma c_i  = c_i(x^1,x^2,x^3),\nonumber
\eeqa
with the function $f$ satisfying
\beqa
\label{eqn:c}
\frac{\partial f}{\partial x^0} = \gN(\nf,\partial_0) = 0 \comma \frac{\partial  f}{\partial x^i} = \gN(\nf,\partial_i) = c_i.
\eeqa
These coordinates, however, are not slice coordinates for the level sets $\Lev_c$; to make them so, simply define new coordinates $(\tilde{x}^i)$ by
\beqa
\left\{
\begin{array}{ll}
\tilde{x}^0 = x^0,\nonumber\\
\tilde{x}^1 = f(x^1,x^2,x^3),\nonumber\\
\tilde{x}^2 = x^2,\nonumber\\
\tilde{x}^3 = x^3.\nonumber\end{array}
\right.
\eeqa
These new coordinates satisfy $\nabla\tilde{x}^1 = \partial /\partial \tilde{x}^0$, and they are indeed slice coordinates for $\Lev_c$:
$$
\Lev_c = \left\{q \in \uu\,:\,\tilde{x}(q) = \left(\tilde{x}^0(q),c,\tilde{x}^2(q),\tilde{x}^3(q)\right)\right\}\cdot
$$
Thus each $N_q^{\perp} = T_q\Lev_c = \text{span}\,\{\partial/\partial \tilde{x}^0|_q,\partial/\partial \tilde{x}^2|_q,\partial/\partial \tilde{x}^3|_q\}$, with
$$
\frac{\partial}{\partial \tilde{x}^0} = \frac{\partial}{\partial x^0}\comma 
\frac{\partial}{\partial \tilde{x}^1} \overset{\eqref{eqn:c}}{=} \frac{1}{c_1}\frac{\partial}{\partial x^1} \comma
\frac{\partial}{\partial \tilde{x}^i}  \overset{\eqref{eqn:c}}{=} -\frac{c_i}{c_1}\,\frac{\partial}{\partial x^1} + \frac{\partial}{\partial x^i},\nonumber
$$
and $\partial/\partial \tilde{x}^0$ has a nice relationship with the other coordinate basis vectors:
$$
\gN(\partial/\partial \tilde{x}^0,\partial/\partial \tilde{x}^1) = 1 \comma \gN(\partial/\partial \tilde{x}^0,\partial/\partial \tilde{x}^i) = -c_i + c_i = 0.
$$
The metric in the new coordinates $(\tilde{x}^i)$ is now precisely in the form of \eqref{nc}.
Finally, note that since $\partial/\partial \tilde{x}^2$ and $\partial/\partial \tilde{x}^3$ are both orthogonal to the lightlike vector $\partial/\partial \tilde{x}^0$, they must satisfy $g_{22}(N) > 0, g_{33}(N) > 0$.  It follows that each embedded 2-submanifold defined by
\beqa
\label{eqn:hpos}
\Lambda_{\scalebox{.5}{\emph{b,c}}} \defeq \left\{q \in \uu\,:\,\tilde{x}(q) = \left(b,c,\tilde{x}^2(q),\tilde{x}^3(q)\right)\right\} \subseteq \Lev_c,
\eeqa
with induced metric $h_{\scalebox{0.5}{\emph{b,c}}} \defeq \sum_{i=2}^3 g_{ij}(b,c,\tilde{x^2},\tilde{x}^3,N)d\tilde{x}^i\otimes d\tilde{x}^j$, is Riemannian; indeed, since $\gN$ has negative determinant,
\beqa
\label{eqn:hpos}
\text{det}\,\gN = -(g_{22}g_{33}-g_{23}^2) < 0 \imp g_{22}g_{33}-g_{23}^2 > 0.
\eeqa
Together with the fact that $g_{22} > 0$, it follows that the two leading submatrices of the $2 \times 2 $ matrix $\begin{pmatrix}g_{22}&g_{23}\\g_{32}&g_{33} \end{pmatrix}$ have positive determinant; thus this submatrix is positive definite. Renaming each $g_{ij}(N)$ to $h_{ij}$ for $i,j=2,3$, the proof is complete.
\end{proof}

\begin{lemma}\label{gradientexist} Given and arbitrary lightlike vector $v_0$ of a Finsler spacetime $(M,L)$, it is possible to extend it to a lightlike gradient vector field in a certain neighborhood of $\pi(v_0)$.
\end{lemma}
\begin{proof}
First, consider a local splitting $I\times B$ of $M$ with $I\subset \RR$ in such a way that $\{t_0\}\times B$ is spacelike for all $t_0\in I$ and $\partial_t$ is timelike. Then consider a surface $S_0$ in $B$ diffeomorphic to $S^2$ which contains $\pi(v_0)$ and is orthogonal to $v_0$, with $v_0$ pointing to the exterior region of $I\times S_0$. Observe that the cylinder $I\times S_0$ is a hypersurface and that it admits a smooth lightlike vector field $N$ along it with the following property: for any $(t_0,p_0)\in I\times S_0$, $N_{(t_0,p_0)}$ is orthogonal to $\{t_0\}\times S_0$. By the Inverse Function Theorem, the exponential map restricted to the bundle generated by $N$ along $S_0$ is a local diffeomorphism. In particular, one can construct local coordinates around $\pi(v_0)$ using product coordinates in $I\times S_0$ and then the one of the exponential map $s\rightarrow \exp_{(t,p)}(sN)$.  All this together implies that the projection onto $I$ in these coordinates provides a function $f\colon\uu\subset M\rightarrow \RR$ whose level sets $f^{-1}(t_0)$ are the hypersurfaces obtained as the union of all the lightlike geodesics passing through the points $(t_0,p)$ with $p\in S_0$ and with velocity $N_{(t_0,p)}$. Reducing $\uu$ if necessary, we can assume that these hypersurfaces coincide with the horismos $E^+(\{t_0\}\times S_0)\cup E^-(\{t_0\}\times S_0)$. Therefore, they are degenerate and the gradient vector field $\nf$ must be lightlike.
\end{proof}
Observe that the last Lemma can be interpreted in the following way. The lightlike gradient vector field can be thought of as the vector field tangent to the light rays departing orthogonally from a given spacelike surface and its temporal cylinder of a given interval of a universal time.

There is a direct relationship between the domain of validity of lightlike coordinates and the existence of focal points along the  geodesic integral curves of $N$.  First, define $\Delta_4 \defeq \sqrt{-\text{det}\,g_{ij}(N)}$ and 
\beqa
    \Delta \defeq \left|
      \begin{array}{cc}
        h_{22} & h_{23}\\
        h_{32} & h_{33}
              \end{array}
    \right|^{\frac{1}{2}},
    \nonumber
\eeqa
and observe that $\Delta_4 = \Delta$.  Then the relationship is as follows:

\begin{lemma}
In lightlike coordinates, a point $p \in \uu$ is a focal point of \eqref{eqn:hpos} along the geodesic integral curve of $N$ through $p$ if and only if $\Delta|_p = 0$.
\end{lemma}

\begin{proof}
Let $(x^0,x^1,x^2,x^3)$ be lightlike coordinates as in the Lemma \ref{theorem:nc}, with  $N = \partial_0$, and let $\Lambda$ be a Riemannian 2-submanifold as in \eqref{eqn:hpos}.  We begin by observing that $\partial_0,\partial_1,\partial_2,\partial_3$ are all Jacobi fields along any geodesic integral curve $\gamma(x^0)$ of $N$ starting in $\Lambda$.  Indeed, setting $J \defeq \partial_i$, and noting that $[J,N] = 0$ and $\nabla^{\scalebox{0.5}{\emph{N}}}_NN = 0$, we have that
$$
J'' = \nabla^{\scalebox{0.5}{\emph{N}}}_N\nabla^{\scalebox{0.5}{\emph{N}}}_N J = \nabla^{\scalebox{0.5}{\emph{N}}}_N\nabla^{\scalebox{0.5}{\emph{N}}}_J N = R_{\scriptscriptstyle N}(N,J)N = -R_{\scriptscriptstyle N}(J,N)N,
$$
where $R_{\scriptscriptstyle N}$ is the $(1,3)$-curvature tensor.  Now suppose that $\Delta|_p = 0$ at a point $p$ along $\gamma$.  Then $\partial_2|_p,\partial_3|_p$ must be linearly dependent, hence some nontrivial linear combination of the two gives the zero vector at $p$.  If we extend this linear combination as is to a vector field $J(x^0)$ along $\gamma$, then $J$ will be a Jacobi field.  In fact it is a $\Lambda$-Jacobi field, since $J(0) \in T_{\gamma(0)}\Lambda$ and since $\text{tan}_{\Lambda} J'(0) = \text{tan}_{\Lambda}(\nabla^{\scalebox{0.5}{\emph{N}}}_{J(0)}N)$; we thus conclude that $p$ is a focal point of $\Lambda$ along $\gamma$.  Now for the converse.  Suppose that $p$ is a focal point of $\Lambda$ along the geodesic integral curve $\gamma(x^0)$ of $N$, with $\gamma(b) = p$; let $J\colon [0,b] \lra \uu$ denote the corresponding $\Lambda$-Jacobi field, with $J(b)$ its (zero) value at $p$:
$$
J(b) = \sum_{i=0}^3 f^i(b)\partial_i\big|_{p} = 0.
$$
(In fact $f^1$ is identically zero, since $J$ is orthogonal to $\gamma$.) If the $f^i(b)$'s are not all zero, then we are done, since we would thus have a nontrivial linear combination of the $\partial_i|_p$'s equalling zero, which can happen only if $\Delta_4|_p = \Delta|_p = 0$.  Thus, assume that each $f^i(b) = 0$, in which case
$$
J'(b) = \dot{f}^i(b)\partial_i|_p + \cancelto{0}{f^i(b)\nabla_{N(b)}^{N(b)}\partial_i},
$$
where at least one of $\dot{f}^i(b)$ is nonzero, for otherwise $J(b) = J'(b) = 0$ and so $J$ would have to be trivial.  Next, because each $\partial_i$ is a Jacobi field, and because any two $\Lambda$-Jacobi fields $V, W$ along a geodesic satisfy
$
\gN(V',W) = \gN(V,W')
$
(see, e.g., \cite[Prop. 3.18]{JavSoa15}), we have that
$$
\underbrace{\,\gN(J',\partial_j)|_p\,}_{\dot{f}^i(b)g_{ij}(N)} = \gN(J,\partial_j')|_p = 0,
$$
the latter because $J(b) = 0$.  Thus we've arrived at the system of equations
$$
g_{ji}(N_p)\dot{f}^i(b) = 0,
$$
with the $\dot{f}^i(b)$'s not all zero.  This can only happen if $\Delta_4|_p = 0$, in which case $\Delta_p = 0$ and the proof is complete.
\end{proof}

\section{Finsler pp-waves and Brinkmann coordinates}
\label{sec-finslerpp}
Before proceeding to Finsler pp-waves, we need to establish when our lightlike gradient vector field is parallel. This is achieved by the following lemma.
\begin{lemma}
	\label{lemma:pp-waves}
	A lightlike gradient vector field $N=\nf$ is parallel if and only if all components of the metric \eqref{nc} are independent of $\tilde{x}^0$. 
\label{lem-lightlike2}
\end{lemma}
\begin{proof}
Assume that $\nabla^{\scalebox{0.5}{\emph{N}}}_XN=0$ for all $X\in \mathfrak{X}(M)$. Consider $\partial_i,\partial_j$ with $i,j=1,\ldots,n$ and observe that by \eqref{nc}, we know that $\gN(N,\partial_i)=0$ if $2\leq i\leq n$ and $\gN(N,\partial_1)=1$. By the Koszul formula \eqref{koszul}, taking into account that $N=\partial_0$ and all the Lie Brackets will vanish,
\[ 0=2 \gN(\nabla^{\scalebox{0.5}{\emph{N}}}_{\partial_i}N,\partial_j)=\partial_i (\gN(N,\partial_j))-\partial_j (\gN(N,\partial_i))+N (\gN(\partial_i,\partial_j)), \]
which implies that $N (\gN(\partial_i,\partial_j))=0$ for all $i,j=1,\ldots, n$, namely, the coefficients $g_{ij}(N)$ do not depend on the coordinate $x^0$.

Assume now that all the $g_{ij}(N)$ do not depend on $x^0$. Let us see first that $\nabla^{\scalebox{0.5}{\emph{N}}}_NN=0$. Using the Koszul formula \eqref{koszul},
\[2 \gN(\nabla^{\scalebox{0.5}{\emph{N}}}_NN,\partial_j)=N (\gN(N,\partial_j))-\partial_j (\gN(N,N))+N (\gN(N,\partial_j))=0\]
for any $j=0,\ldots,n$, which implies that $\nabla^{\scalebox{0.5}{\emph{N}}}_NN=0$. Using this identity and the Koszul formula once more,
\[ 2 \gN(\nabla^{\scalebox{0.5}{\emph{N}}}_{\partial_i}N,\partial_j)=\partial_i (\gN(N,\partial_j))-\partial_j (\gN(N,\partial_i))+N (\gN(\partial_i,\partial_j))=0, \]
for any $i,j=0,\ldots,n$, which implies that $\nabla^{\scalebox{0.5}{\emph{N}}}_{\partial_i}N=0$.
\end{proof}

In the following, given an $\bar A$-admissible $N$ defined in some open subset $\uu\subset M$, we will use the notation $ \Gamma(N^{\perp})$ for the space of vector fields $X\in\mathfrak{X}(\uu)$ such that $\gN(N,X)=0$.
\begin{thm}[pp-waves; \cite{globke}]
	\label{thm:pp-waves}
If $N$ is a parallel lightlike gradient vector field, then the curvature endomorphism $R_{\scriptscriptstyle N}$  satisfies
	\beqa
	\label{eqn:ppwave}
	R_{\scriptscriptstyle N}(X,Y)Z = 0,
	\eeqa
	for all $X, Y, Z \in \Gamma(N^{\perp})$, if and only if there exist local coordinates $(v,u,x,y)$ in which \eqref{nc} takes the form
	\beqa
	\label{pp-wave1}
	(g_{ij}(N)) = 
	\left(
	\begin{array}{cccc}
		0 & 1 & 0 & 0\\
		1 &  H & 0 & 0\\
		0 & 0 & 1 & 0\\
		0 & 0 & 0 & 1
	\end{array}
	\right),
\label{eq-brinkmann}
	\eeqa
	where $g_{uu}(N) \defeq H(u,x,y)$. Such a metric is called a \emph{Finsler pp-wave (}expressed in so-called \emph{Brinkmann coordinates \cite{brinkmann})}.
\end{thm}

\begin{proof}
Consider each embedded 2-submanifold $\Lambda_{\scalebox{.5}{\emph{b,c}}}$ given by \eqref{eqn:hpos}, with its corresponding induced Riemannian metric 
	$$
	\hc \defeq \sum_{i,j=2}^3h_{ij}(c,\tilde{x}^2,\tilde{x}^3)d\tilde{x}^i\otimes d\tilde{x}^j.
	$$
	(Since by Lemma \ref{lemma:pp-waves} each $g_{ij}(N)$ is independent of $\tilde{x}^0$, the components $h_{ij}$ are $\tilde{x}^0$-independent.) By the Gauss Equation (\cite[Theorem~8.5,~p.~230]{Lee}), we know that the (single) component $\text{Rm}_{\scalebox{.5}{\emph{c}}}(\partial_2,\partial_3,\partial_2,\partial_3)$ of the curvature tensor of $\hc$ is related to that of  $g$ by
	\beqa
	\overbrace{\,\text{Rm}_N(\partial_2,\partial_3,\partial_2,\partial_3)\,}^{\text{0, by \eqref{eqn:ppwave}}} &=& \text{Rm}_{\scalebox{.5}{\emph{b,c}}}(\partial_2,\partial_3,\partial_2,\partial_3)\label{eqn:Riem0}\\
	&&\hspace{-.5in} -\, g_{\scriptscriptstyle N}(\sff_{\scalebox{.5}{\emph{b,c}}}(\partial_2,\partial_3),\sff_{\scalebox{.5}{\emph{b,c}}}(\partial_3,\partial_2)) + g_{\scriptscriptstyle N}(\sff_{\scalebox{.5}{\emph{b,c}}}(\partial_2,\partial_2),\sff_{\scalebox{.5}{\emph{b,c}}}(\partial_3,\partial_3)),\nonumber
	\eeqa
	where $\sff_{\scalebox{.5}{\emph{b,c}}}$ is the second fundamental form of $\hc$; observe that as $N$ is parallel, the curvature tensor of $\gN$ coincides with $R_{\scriptscriptstyle N}$.  But when $N=\partial_0$ is parallel, this equation simplifies considerably, because in such a case the vector field $\nabla^{\scalebox{0.5}{\emph{N}}}_{\partial_2}{\partial_3}$ has no $\partial_1$-component,
	\beqa
	\label{eqn:Pi1}
	\nabla^{\scalebox{0.5}{\emph{N}}}_{\partial_2}{\partial_3} = \alpha\partial_0+\beta\partial_2+\gamma\partial_3,
	\eeqa
	for some smooth functions $\alpha,\beta,\gamma$ (here use that 
	\[\gN(\nabla^{\scalebox{0.5}{\emph{N}}}_{\partial_2}{\partial_3} ,\partial_0)=-\gN(\nabla^{\scalebox{0.5}{\emph{N}}}_{\partial_2}{\partial_0} ,\partial_3)=0\] because $\partial_0=N$ is parallel and the compatibility of $\nabla$ and $g$). Thus, since $\partial_0$ is orthogonal to $\Lambda_{\scalebox{.5}{\emph{b,c}}}$,
	\beqa
	\label{eqn:Pi2}
	\sff_{\scalebox{.5}{\emph{b,c}}}(\partial_2,\partial_3) \defeq (\nabla^{\scalebox{0.5}{\emph{N}}}_{\partial_2}{\partial_3})^{\perp} = \alpha\partial_0,
	\eeqa
	where, as $N$ is parallel, we observe that $\nabla^{\scalebox{0.5}{\emph{N}}}$ is the Levi-Civita connection of $\gN$;
	likewise for $\sff_{\scalebox{.5}{\emph{b,c}}}(\partial_2,\partial_2)$ and $\sff_{\scalebox{.5}{\emph{b,c}}}(\partial_3,\partial_3)$.\footnote{Since the subspace $S \defeq \text{span}\{\partial_2,\partial_3\}$ is spacelike, its orthogonal complement $S^{\perp}$ is timelike, and each tangent space $T_p\uu$ is the direct sum $S_p \oplus S_p^{\perp}$ (see, e.g., \cite[p.~141]{o1983}).  Therefore, $\alpha \partial_0 \in S^{\perp}$ is the (unique) normal component of $\cd{\partial_2}{\partial_3} = \alpha\partial_0+\beta\partial_2+\gamma\partial_3$.} Thus, as $\partial_0$ is lightlike, \eqref{eqn:Riem0} simplifies to
	\beqa
	\label{eqn:Pi}
	\text{Rm}_{\scalebox{.5}{\emph{b,c}}} = 0.
	\eeqa
In other words, each Riemannian submanifold $(\Lambda_{\scalebox{.5}{\emph{b,c}}},\hc)$ is flat, which means that there exist local coordinates $r_{\scalebox{.5}{\emph{c}}} \defeq r_{\scalebox{.5}{\emph{c}}}(\tilde{x}^2,\tilde{x}^3),s_{\scalebox{.5}{\emph{c}}} \defeq s_{\scalebox{.5}{\emph{c}}}(\tilde{x}^2,\tilde{x}^3)$ on each $\Lambda_{\scalebox{.5}{\emph{b,c}}}$ with respect to which the induced metric $h_{\scalebox{.5}{\emph{b,c}}}$ takes the form
$$
h_{\scalebox{.5}{\emph{b,c}}} = dr_{\scalebox{.5}{\emph{c}}}\otimes  dr_{\scalebox{.5}{\emph{c}}} + ds_{\scalebox{.5}{\emph{c}}}\otimes ds_{\scalebox{.5}{\emph{c}}}.
$$
At each $\tilde{x}^1 = c$, we thus have the triple $(c,r_{\scalebox{.5}{\emph{c}}},s_{\scalebox{.5}{\emph{c}}})$; considering $-s_{\scalebox{.5}{\emph{c}}}$ if necessary, we may also assume that each $\{\partial_{r_{\scalebox{0.5}{\emph{c}}}},\partial_{s_{\scalebox{0.5}{\emph{c}}}}\}$ is positively oriented.  We now put these together to form a  smooth coordinate chart $(\tilde{x}^0,\tilde{x}^1,x,y)$, as follows.  First, at each point $(b,c,0,0) \in \Lambda_{\scalebox{.5}{\emph{b,c}}}$, rotate each $\{\partial_{r_{\scalebox{0.5}{\emph{c}}}},\partial_{s_{\scalebox{0.5}{\emph{c}}}}\}$ so that $\partial_{r_{\scalebox{0.5}{\emph{c}}}}$ points in the direction of $\partial_2$ (i.e., $\partial_{r_{\scalebox{0.5}{\emph{c}}}} = \frac{1}{\sqrt{h_{22}}}\partial_2\big|_{(b,c,0,0)}$); these rotated $\partial_{r_{\scalebox{0.5}{\emph{c}}}}$'s thus comprise a smooth vector field $V \defeq \frac{1}{\sqrt{h_{22}}}\partial_2$ on the submanifold $\Sigma \defeq \{(\tilde{x}^0,\tilde{x}^1,0,0)\}$.  As each $\partial_{s_{\scalebox{0.5}{\emph{c}}}}$ is orthogonal to its corresponding $\partial_{r_{\scalebox{0.5}{\emph{c}}}}$ and since an orientation has been fixed, the $\partial_{s_{\scalebox{0.5}{\emph{c}}}}$'s thus also comprise a smooth vector field $W$ on $\Sigma$, namely, the unique unit length vector field orthogonal to $V$ and such that $\{V,W\}$ is positively oriented.  Next, at each $p \in \Sigma \cap \Lambda_{\scalebox{.5}{\emph{b,c}}}$, parallel transport $V,W$ along the integral curve $\gamma^{(p)}$ of $\partial_{r_{\scalebox{0.5}{\emph{c}}}}$ through $p$, via the connection $\nabla^{\scalebox{0.5}{\emph{b,c}}}$ compatible with the induced metric $h_{\scalebox{.5}{\emph{b,c}}}$; then, at each point along $\gamma^{(p)}$, parallel transport along the integral curve of $\partial_{s_{\scalebox{0.5}{\emph{c}}}}$ on $\Lambda_{\scalebox{.5}{\emph{b,c}}}$; by abuse of notation, let $V,W$ denote the resulting vector fields, now smoothly defined on $\uu$.  By the Gauss Formula (see, e.g., \cite[Theorem~8.2,~p.~228]{Lee}) and the flatness condition \eqref{eqn:Pi}, we have, on each $(\Lambda_{\scalebox{.5}{\emph{b,c}}},\hc)$, that
\[
 \nabla^{\scalebox{0.5}{\emph{N}}}_V V = \cancelto{0}{\nabla^{\scalebox{0.5}{b,c}}_{\!V}\,V} + \sff_{\scalebox{.5}{\emph{b,c}}}(V,V);
 \]
similarly for $\nabla^{\scalebox{0.5}{\emph{N}}}_V W, \nabla^{\scalebox{0.5}{\emph{N}}}_W V$, and $\nabla^{\scalebox{0.5}{\emph{N}}}_W W$.  But just as in \eqref{eqn:Pi2}, each $\sff_{\scalebox{.5}{\emph{b,c}}}(\cdot,\cdot)$ must point solely in the direction of $\partial_0$.  Finally, define the orthonormal pair
$$
X \defeq -g(\partial_1,V)\partial_0 + V \comma Y \defeq -g(\partial_1,W)\partial_0 + W,
$$
which pair is now orthogonal to $\partial_0,\partial_1$.  It follows that all of the Lie brackets of the frame $\{\partial_0,\partial_1,X,Y\}$ will have only a $\partial_0$-component, which collectively yields that $X^{\flat} = g(X,\cdot)$ and $Y^{\flat} = g(Y,\cdot)$ will be closed; indeed, for any pair of vector fields $A,B \in \{\partial_0,\partial_1,X,Y\}$,
$$
dX^{\flat}(A,B) = \cancelto{0}{A(g(X,B))}-\cancelto{0}{B(g(X,A))} - \cancelto{0}{g(X,[A,B])} = 0;
$$
likewise with $dY^{\flat}$. By the Poincar\'e Lemma, both 1-forms are locally exact: $X^{\flat} = dx$ and $Y^{\flat} = dy$, for some smooth functions $x,y$. With respect to the coordinate chart $(\tilde{x}^0,\tilde{x}^1,x,y)$, the ambient Lorentzian metric $g$ thus takes the form
	\beqa
	(g_{ij}(N)) = 
	\left(
	\begin{array}{cccc}
		0 & 1 & 0 & 0\\
		1 & g_{11}(N) & 0 & 0\\
		0 & 0 & 1 & 0\\
		0 & 0 & 0 & 1
	\end{array}
	\right),\nonumber
	\eeqa
	which is precisely \eqref{pp-wave1}. (This argument generalizes to dimensions $> 4$.)
	Conversely, suppose that local coordinates $(v,u,x,y)$ exist in which the metric takes this form, with $N = \partial_v = \nabla u$.  Then the nonvanishing Christoffel symbols of such a metric are
	\beqa
	\nabla^{\scalebox{0.5}{\emph{N}}}_{\partial_x}{\partial_u} \!&=&\! \nabla^{\scalebox{0.5}{\emph{N}}}_{\partial_u}{\partial_x}= \frac{H_x}{2}\partial_v,\nonumber\\
	\nabla^{\scalebox{0.5}{\emph{N}}}_{\partial_y}{\partial_u} \!&=&\! \nabla^{\scalebox{0.5}{\emph{N}}}_{\partial_u}{\partial_y}= \frac{H_y}{2}\partial_v,\nonumber\\
	\nabla^{\scalebox{0.5}{\emph{N}}}_{\partial_u}{\partial_u} \!&=&\! \frac{H_u}{2}\partial_v - \frac{H_x}{2}\partial_x - \frac{H_y}{2}\partial_y,\nonumber
	\eeqa
	from which it follows that $R_{\scriptscriptstyle N}(\partial_x,\partial_y)\partial_x = R_{\scriptscriptstyle N}(\partial_x,\partial_y)\partial_y = 0$; this is precisely the curvature condition \eqref{eqn:ppwave}.  Actually something else vanishes, too:
	$$
	R_{\scriptscriptstyle N}(\partial_x,\partial_y)\partial_u =  \nabla^{\scalebox{0.5}{\emph{N}}}_{\!\partial_x}\nabla^{\scalebox{0.5}{\emph{N}}}_{\!\partial_y}\partial_u - \nabla^{\scalebox{0.5}{\emph{N}}}_{\!\partial_y}\nabla^{\scalebox{0.5}{\emph{N}}}_{\!\partial_x}\partial_u = \frac{H_{yx}}{2}\partial_v - \frac{H_{xy}}{2}\partial_v = 0. 
	$$
	Thus $R_{\scriptscriptstyle N}(X,Y)W = 0$ for all $X,Y \in \Gamma(N^{\perp})$ and $W \in \mathfrak{X}(\uu)$.
\end{proof}
We will give some examples of Finsler pp-waves in Section \ref{ExFinslerppwaves}, but let us observe that the examples of Finsler pp-waves already present in literature can be included in the definition of Theorem \ref{thm:pp-waves} up to some interpretations. In both cases, \cite{fusterpabst,HPF21}, the Lorentz-Finsler metric is not smooth on the lightlike parallel vector field $N$, but as these Finsler spacetimes are of Berwald type, there is an available affine connection which makes $N$ parallel. Moreover, the role of $\gN$-orthogonality to $N$ can be played by the tangent space to the lightcone at $N$. In both cases, it coincides with the tangent space to the lightlike cone of a Lorentzian metric which is already a pp-wave with $N$ as parallel vector field. Therefore, the conditions of Theorem \ref{thm:pp-waves} are satisfied also for the Berwaldian Finsler pp-waves of \cite{fusterpabst,HPF21}.

\section{The quotient bundle of a Finsler pp-wave}
\label{sec:bundle}
\begin{thm}[\cite{caja,leistner}]\label{thm:QB}
	Let $(M,L)$ be a Finsler spacetime and $N = \nf$ a lightlike, parallel gradient vector field defined in an open subset $\uu \subseteq M$, with orthogonal complement $N^{\perp}  \subseteq T\uu$.  Then the vector bundle $N^{\perp}/N$ admits a positive-definite inner product $\bar{g}$,
	$$
	\bar{g}([X],[Y]) \defeq g_{\scalebox{0.5}{N}}(X,Y)\hspace{.2in}\text{for all}\hspace{.2in}[X], [Y] \in \Gamma(N^{\perp}/N),
	$$
	and a corresponding linear connection $\overline{\nabla}\colon \mathfrak{X}(\uu)\times \Gamma(N^{\perp}/N) \lra \Gamma(N^{\perp}/N)$,
	$$
	\conn{W}{Y} \defeq [\nabla^{\scalebox{0.5}{N}}_{W}{Y}] \hspace{.2in}\text{for all}\hspace{.2in} W \in \mathfrak{X}(\uu)\hspace{.2in}\text{and}\hspace{.2in}[Y] \in \Gamma(N^{\perp}/N).
	$$
	This connection is flat if and only if $(\uu,L|_{\uu})$ is a Finsler pp-wave.
\end{thm}

\begin{proof}
	The metric $\bar{g}$ will be well defined, and positive definite, whenever $N$ is lightlike; indeed, every $X \in \Gamma(N^{\perp})$ not proportional to $N$ is necessarily spacelike, so that $\bar{g}$ is nondegenerate (and positive-definite), and if $[X] = [X']$ and $[Y] =[Y']$, so that  $X' = X +fN$ and $Y' = Y+kN$ for some smooth functions $f,k$, then
	$$
	\bar{g}([X'],[Y']) = \gN(X',Y') = \gN(X,Y) =  \bar{g}([X],[Y]).
	$$
	On the other hand, the connection $\overline{\nabla}$ requires $N$ to be parallel or else it is not well defined: $\nabla^{\scalebox{0.5}{\emph{N}}}_{W}{Y}  \in \Gamma(N^{\perp})$ if and only if $N$ is parallel, in which case
	$$
	\conn{W}{Y'} = [\nabla^{\scalebox{0.5}{\emph{N}}}_{W}{Y'}] = [\nabla^{\scalebox{0.5}{\emph{N}}}_{W}{Y}]+\cancelto{0}{[W(k)N]} + \cancelto{0}{[k\nabla^{\scalebox{0.5}{\emph{N}}}_{W}{N}]}\, = \conn{W}{Y}.
	$$
	That $\overline{\nabla}$ is indeed a linear connection follows easily.  Now, if this connection is flat, then by definition its curvature endomorphism
	\beqa
	\overline{\text{R}}\colon \mathfrak{X}(\uu)  \times \mathfrak{X}(\uu) &\times& \Gamma(N^{\perp}/N) \rightarrow \Gamma(N^{\perp}/N),~\text{given by}\nonumber\\
	\overline{\text{R}}(U,W)[X] &\defeq& \overline{\nabla}_{\!V}[\overline{\nabla}_{\!W}[X]] - \overline{\nabla}_{\!W}[\overline{\nabla}_{\!U}[X]] - \conn{[U,W]}{X}\nonumber
	\eeqa
	vanishes, for any section $[X] \in \Gamma(N^{\perp}/N)$ and vector fields $U,W \in \mathfrak{X}(\uu)$. Using the metric $\bar{g}$, this flatness condition is equivalent to
	$$
	\bar{g}(\overline{\text{R}}(U,W)[X],[Y]) = 0\hspace{.1in}\text{for all}\hspace{.1in}U,W \in \mathfrak{X}(\uu)\hspace{.1in},\hspace{.1in}[X],[Y] \in \Gamma(N^{\perp}/N).
	$$
	But if we unpack the definitions of $\overline{\nabla}$ and $\bar{g}$, we see that
	\beqa
	\label{eqn:RtoR}
	\bar{g}(\overline{\text{R}}(U,W)[X],[Y]) = \text{Rm}_{\scalebox{0.5}{\emph{N}}}(U,W,X,Y) = \text{Rm}_{\scalebox{0.5}{\emph{N}}}(X,Y,U,W).
	\eeqa
	It follows that $\overline{\text{R}} = 0$ if and only if $R_{\scriptscriptstyle N}(X,Y)W = 0$ for all $X,Y \in \Gamma(N^{\perp})$ and $W \in \mathfrak{X}(\uu)$; by \eqref{eqn:ppwave}, this is precisely the condition to be a Finsler pp-wave.
\end{proof}

\section{Examples}
\label{ExFinslerppwaves}
\label{sec:Examples}
\subsection{Parallel lightlike vector field}
Let us choose a Finsler metric $F$ on $\RR^2\times M$ with $v,u$ the coordinates of $\RR^2$ and such that $\partial_v$ is a Killing field, namely, $F$ does not depend on $v$. Define a one-form $\omega$ such that
\begin{align*}
\omega(\partial_v)&=F(\partial_v),\\
\omega(\partial_u)&=(1+\gN(\partial_u,\partial_v))/F(\partial_v),\\
\omega(\partial_x)&=\gN(\partial_v,\partial_x)/F(\partial_v),
\end{align*}
where $\partial_x$ is tangent to $M$. Then the Lorentz-Finsler metric defined by
\begin{equation}\label{ExParallel}
L(v)\defeq \omega(v)^2-F(v)^2
\end{equation}
admits $N=\partial_v$ as a lightlike parallel vector field (here we are using the fact that if $A=\{v\in T(\RR^2\times M): L(v)>0\}$ is nonempty, then $L$ determines a Finsler spacetime on $\RR^2\times M$, see \cite[Theorem 4.1]{JS20}). To check this, observe that the fundamental tensor $g^{\scalebox{0.5}{\emph{L}}}$ of $L$ is given by
\[g^{\scalebox{0.5}{\emph{L}}}_v(u,w)=\omega(u)\omega(w)-g_{\scalebox{0.5}{\emph{V}}}(u,w)\]
where $g$ is the fundamental tensor of $F$. By the definition of $\omega$, it follows that the coefficients $g^{\scalebox{0.5}{\emph{L}}}_{ij}(N)$, with $N=\partial_v$, form a matrix as in \eqref{nc}. This implies in particular that $N=\partial_v$ is the $L$-gradient of the function $f\colon\RR^2\times M\rightarrow \RR$ defined as $f(v,u,p)=u$. By Lemma \ref{lemma:pp-waves},  recalling that $F$ has been chosen independent of $v$ and that $\omega$ has been constructed from $F$, we conclude that $N$ is parallel.

\subsection{Finsler pp-waves}
Let us give a particular example of Finsler pp-wave. Choose an arbitrary Finsler metric $F$ on $\RR^2$ (depending on $(u,x,y)$) and define the one-form in $\RR^2$ determined by
\begin{align*}
\omega(\partial_v)&=F(\partial_v),\\
\omega(\partial_u)&=(1+\gN(\partial_u,\partial_v))/F(\partial_v),
\end{align*}
then 
\begin{equation}\label{Finslerppwave}
L=\omega^2-F^2-dx^2-dy^2
\end{equation} is a Finsler pp-wave defined in the region $A\{v\in T(\RR^2\times M): L(v)\geq 0\}$. Observe that the matrix of $\gN$, for $N=\partial_v$, is of the form \eqref{pp-wave1}. Moreover, this metric is smooth on $A$, because it is not smooth in the vectors $(0,w)\in T(\RR^2\times M)$ with $w\in TM$ but in this case, $L(0,w)=-w_1^2-w^2_2<0$ if $w\not=0$. Even if this metric is not of the same type as \eqref{ExParallel}, it is very similar. Indeed, in this case, the role of the metric $F$ in \eqref{ExParallel} is played by $\sqrt{F^2+dx^2+dy^2}$, which is not a regular Finsler metric because of the smoothness issues, but the proof of  \cite[Theorem 4.1]{JS20}) still works to guarantee that \eqref{Finslerppwave} determines a Finsler spacetime. Moreover, as in the last subsection, one can check that $L$ admits $N=\partial_v$ as a parallel and lightlike vector field, and the coordinates of $\gN$ are of the form \eqref{pp-wave1}, which concludes by Theorem \ref{thm:pp-waves}, that $L$ is a Finsler pp-wave.

\section{Penrose's construction}
\label{sec-penrose}
With the generalizations of lightlike coordinates and pp-waves to Finsler spacetimes in Sections \ref{sec-lightlike} and \ref{sec-finslerpp}, respectively, we are now ready to consider the extension of Penrose's ``plane wave limit" \cite{penPW}. Again, consider a 4-dimensional Finsler spacetime $(M,L)$ and $N$ a lightlike gradient vector field such that locally in lightlike coordinates $(x^0,\dots,x^3)$, $g_{\scriptscriptstyle N}$ takes the form
\beqa
\label{eqn:lor*}
(g_{ij}(N)) \defeq
    \begin{pmatrix}
        0 & 1  & 0 & 0 \\
        1 & g_{11}(N) & g_{12}(N) & g_{13}(N) \\
        0 & g_{21}(N) &  h_{22} & h_{23} \\
        0 & g_{31}(N) &  h_{32} & h_{33}
      \end{pmatrix},\nonumber
\eeqa
being $N=\frac{\partial }{\partial x^0}$. Recall that lightlike gradient vector fields always exist at least locally (see Lemma \ref{gradientexist}).
Penrose's construction now proceeds as follows.  Define another coordinate system $(\tilde{x}^0,\ldots,\tilde{x}^3)$ via the diffeomorphism $\varphi$ given by
\beqa
\label{eqn:tilde}
(x^0,\dots,x^3) \overset{\varphi}{\mapsto} (x^0,\Omega^{-2}x^1,\Omega^{-1} x^2,\Omega^{-1} x^3) \defeq (\tilde{x}^0,\dots,\tilde{x}^3),\nonumber\\
\eeqa
where $\Omega > 0$ is a constant.  Next, define the following metric $\gpw$ in the new coordinates $(\tilde{x}^0,\dots,\tilde{x}^3)$,
\beqa
\label{newmetric}
\big((\gpw)_{ij}\big) \defeq 
    \underbrace{\,\begin{pmatrix}
        0 & 1  & 0 & 0 \\
        1 & \Omega^2g_{11}(N) & \Omega g_{12}(N) & \Omega g_{13}(N) \\
        0 & \Omega g_{21}(N) &  h_{22} & h_{23}
\\
        0 & \Omega g_{31}(N) &  h_{32} & h_{33}
      \end{pmatrix}\,}_{\text{defined in the coordinates}~(\tilde{x}^0,\tilde{x}^1,\tilde{x}^2,\dots,\tilde{x}^n)},\nonumber
\eeqa
where each component $(\gpw)_{ij}$ is (strategically) defined as follows,
\beqa
(\gpw)_{11}(\tilde{x}^0,\dots,\tilde{x}^3) &\defeq& \Omega^2 \underbrace{\,{g}_{11}(N)(\tilde{x}^0,\Omega^2\tilde{x}^1,\Omega\,\tilde{x}^2,\Omega\,\tilde{x}^3)\,}_{=\,{g}_{11}(N)(x^0,\dots,x^3)}\label{newcomp},\\
(\gpw)_{22}(\tilde{x}^0,\dots,\tilde{x}^3) &\defeq&\underbrace{\,h_{22}(\tilde{x}^0,\Omega^2\tilde{x}^1,\Omega\,\tilde{x}^2,\Omega\,\tilde{x}^3)\,}_{=\,{h}_{22}(x^0,\dots,x^3)}\label{newcomp2},
\eeqa
and similarly with the others.  Note that as $\Omega \to 0$,
\beqa
\lim_{\Omega \to 0}  (\gpw)_{11} &\overset{\eqref{newcomp}}{=}& 0\cdot {g}_{11}(\tilde{x}^0,0,\dots,0) = 0,\nonumber\\
\lim_{\Omega \to 0}  (\gpw)_{22} &\overset{\eqref{newcomp2}}{=}& {h}_{22}(\tilde{x}^0,0,\dots,0);\label{newcomp4}\nonumber
\eeqa
similarly with the others.  Note also that the metric $\gpw$ is \emph{conformal} to $(\varphi^{-1})^*\gN$; to see this, use the fact that 
$$
((\varphi^{-1})^*\gN)(\partial_{\tilde{x}^i},\partial_{\tilde{x}^j})d\tilde{x}^i \otimes d\tilde{x}^j = \gN(\partial_{x^i},\partial_{x^j})dx^i\otimes dx^j,
$$
as well as \eqref{eqn:tilde}, to obtain
\beqa
dx^0 \otimes dx^1 &=& \Omega^{2}\,d\tilde{x}^0 \otimes d\tilde{x}^1, \nonumber\\
g_{11}(N)(x^0,x^1,x^2,\dots,x^n)\,dx^1 \otimes dx^1 \!\!&\overset{\eqref{newcomp}}{=}&\!\! \Omega^{2}\,(\gpw)_{11}(\tilde{x}^0,\tilde{x}^1,\tilde{x}^2,\dots,\tilde{x}^n)\,d\tilde{x}^1 \otimes d\tilde{x}^1,\nonumber\\
g_{12}(N)(x^0,x^1,x^2,\dots,x^n)\,dx^1 \otimes dx^2 &=& \Omega^{2}\,(\gpw)_{12}(\tilde{x}^0,\tilde{x}^1,\tilde{x}^2,\dots,\tilde{x}^n)\,d\tilde{x}^1 \otimes d\tilde{x}^2,\nonumber\\
h_{22}(N)(x^0,x^1,x^2,\dots,x^n)\,dx^2 \otimes dx^2 \!\!&\overset{\eqref{newcomp2}}{=}&\!\!\Omega^{2}\,(\gpw)_{22}(\tilde{x}^0,\tilde{x}^1,\tilde{x}^2,\dots,\tilde{x}^n)\,d\tilde{x}^2 \otimes d\tilde{x}^2,\nonumber
\eeqa
and so on, which clearly yields the relationship
\beqa
\label{eqn:cm}
(\varphi^{-1})^*\gN = \Omega^{2}\,\gpw.
\eeqa
In particular, setting $\tilde{g} \defeq (\varphi^{-1})^*\gN$, observe that the homothety \eqref{eqn:cm} means that the Levi-Civita connections of $\gpw$ and $\tilde{g}$ are equal: $\nabla^{\scriptscriptstyle \Omega} = \nabla^{\scriptscriptstyle \tilde{g}}$. Taking the limit
$$
\lim_{\Omega \to 0} \gpw = \lim_{\Omega \to 0} \frac{(\varphi^{-1})^*\gN}{\Omega^{2}}
$$
thus yields the metric
\beqa
\label{plw2}
    \underbrace{\,\begin{pmatrix}
        0 & 1  & 0 & 0 \\
        1 & 0 & 0 & 0 \\
        0 & 0 &  h_{22}(\tilde{x}^0,0,\dots,0) & h_{23}(\tilde{x}^0,0,\dots,0) \\
        0 & 0 &  h_{32}(\tilde{x}^0,0,\dots,0) & h_{33}(\tilde{x}^0,0,\dots,0)
      \end{pmatrix}\,}_{\text{defined in the coordinates}~(\tilde{x}^0,\dots,\tilde{x}^3)},\nonumber
\eeqa
which is clearly a metric in so-called Rosen coordinates of Lorentzian geometry, also referred to as Baldwin-Jeffery-Rosen coordinates (e.g. \cite{duvalgibbons}) due to the pioneering paper \cite{baldwinjeffery} by Baldwin and Jeffery. So, finally, we need to confirm that this can indeed be interpreted as a Finsler pp-wave in Brinkmann coordinates as discussed before. Following the standard coordinate transformation as described e.g. by Blau and O'Loughlin \cite{blauoloughlin}, we introduce a matrix $M$ such that
\[
h_{ij}(\tilde{x}_0)M^i_k(\tilde{x}^0) M^j_l(\tilde{x}^0)=\delta_{kl}\,,
\]
satisfying the symmetry property
\[
h_{ij} \dot{M}^i_k M^j_l
=
h_{ij}M^i_k \dot{M} ^j_l\,,
\]
with $2\leq i\,,j\,,k\leq 3$ labelling the transverse coordinates, and the dot denoting differentiation with respect to $\tilde{x_0}$. Since $h_{ij}$ depends smoothly on $\tilde{x}_0$, such a matrix $M$ can always be found for sufficiently small $\tilde{x}_0$ (e.g. \cite{blau}). Then the transformation
\begin{align*}
\tilde{x}^0&=u\,, \\
\tilde{x}^1&=v-\frac{1}{2}h_{ij}\dot{M}^i_k M^j_l x^k x^l\,,\\
\tilde{x}^i&=M^i_jx^j\,,
\end{align*}
gives rise to the metric in Brinkmann coordinates $(v\,,u\,,x^i)$ with
\[
H=-(h_{ij} \dot{M}^i_k\dot{)} M^j_l x^k x^l
\]
as the $uu$-component of $\gN$ in (\ref{eq-brinkmann}).
\section{Concluding remarks}
The first main result of this article is the construction of a lightlike gradient vector field $N$ in Finsler spacetimes (see Lemma \ref{gradientexist}) yielding a particular chart for $\gN$, as given in Lemma \ref{lem-lightlike}. Secondly, after establishing a condition for $N$ to be parallel, we extend the notion of pp-waves, as given by \cite{globke}, to Finsler spacetimes in Theorem \ref{thm:pp-waves}. New examples of such Finsler pp-waves are found in Section \ref{sec:Examples}, and we also show their quotient bundle structure in Section \ref{sec:bundle}. Finally, Penrose's plane wave limit is adapted to Finsler pp-waves in Section \ref{sec-penrose}.
\newline
\indent In closing, let us note that Lorentzian spacetimes with a parallel lightlike vector field are also known as \emph{Bargmann manifolds}. These have proven useful for studying kinematical groups, in particular the Carroll group of plane waves (notably in \cite{duvalgibbons}), thus raising the question of the kinematical group structure associated with Finsler pp-waves. The optical properties of Finsler pp-waves may offer another worthwhile avenue for future work, as Lorentzian plane waves are well-known to exhibit some remarkable lensing effects (e.g., \cite{harte}). Indeed, it may be interesting to see whether Finsler pp-waves can also be regarded as members of a Kundt class generalized to Finsler spacetimes.

\section*{Acknowledgements}
MAJ was supported by the project  PGC2018-097046-B-I00 funded by MCIN/ AEI /10.13039/501100011033/ FEDER ``Una manera de hacer Europa.''

\bibliographystyle{alpha}
\bibliography{finslerpp}

\newcommand{\etalchar}[1]{$^{#1}$}
\begin{thebibliography}{DGHZ17}

\bibitem[AJ16]{AJ16}
Amir~Babak Aazami and Miguel~\'Angel Javaloyes.
\newblock Penrose's singularity theorem in a {F}insler spacetime.
\newblock {\em Classical Quantum Gravity}, 33(2):025003, 22, 2016.

\bibitem[Asa85]{Asanov}
G.~S. Asanov.
\newblock {\em Finsler geometry, relativity and gauge theories}.
\newblock Fundamental Theories of Physics. D. Reidel Publishing Co., Dordrecht,
  1985.

\bibitem[Bee70]{Beem71}
John~K. Beem.
\newblock Indefinite {F}insler spaces and timelike spaces.
\newblock {\em Canadian J. Math.}, 22:1035--1039, 1970.

\bibitem[BFOP02]{blau}
Matthias Blau, Jose Figueroa-O'Farrill, and George Papadopoulos.
\newblock Penrose limits, supergravity and brane dynamics.
\newblock {\em Classical and Quantum Gravity}, 19(18):4753, 2002.

\bibitem[BJ26]{baldwinjeffery}
O.~R. Baldwin and G.~B. Jeffery.
\newblock The relativity theory of plane waves.
\newblock {\em Proc. R. Soc. Lond. A}, 111:95--104, 1926.

\bibitem[BJS20]{BJS20}
Antonio~N. Bernal, M.~A. Javaloyes, and M.~S\'anchez.
\newblock Foundations of {F}insler {S}pacetimes from the {O}bservers'
  {V}iewpoint.
\newblock {\em Universe}, 6(4), 2020.

\bibitem[Bla11]{blau2}
Matthias Blau.
\newblock Plane waves and {P}enrose limits.
\newblock {\em Lecture Notes for the ICTP School on Mathematics in String and
  Field Theory (June 2-13 2003)}, 2011.

\bibitem[BO03]{blauoloughlin}
Matthias Blau and Martin O'Loughlin.
\newblock {Homogeneous plane waves}.
\newblock {\em Nucl. Phys. B}, 654:135--176, 2003.

\bibitem[Bri25]{brinkmann}
H.W. Brinkmann.
\newblock Einstein spaces which are mapped conformally on each other.
\newblock {\em Mathematische Annalen}, 94(1):119--145, 1925.

\bibitem[CBS13]{caja}
Miguel~S{\'a}nchez Caja, Oihane~F. Blanco, and Jos{\'e} M.~M. Senovilla.
\newblock Structure of second-order symmetric {L}orentzian manifolds.
\newblock {\em Journal of the European Mathematical Society}, 15(2):595--634,
  2013.

\bibitem[CS18]{CaStan}
Erasmo Caponio and Giuseppe Stancarone.
\newblock On {F}insler spacetimes with a timelike {K}illing vector field.
\newblock {\em Classical Quantum Gravity}, 35(8):085007, 28, 2018.

\bibitem[DGHZ17]{duvalgibbons}
C.~Duval, G.~W. Gibbons, P.~A. Horvathy, and P.~M. Zhang.
\newblock {Carroll symmetry of plane gravitational waves}.
\newblock {\em Class. Quant. Grav.}, 34(17):175003, 2017.

\bibitem[EK62]{ehlerskundt}
J{\"u}rgen Ehlers and Wolfgang Kundt.
\newblock Exact solutions of the gravitational field equations.
\newblock In {\em Gravitation: an Introduction to {C}urrent {R}esearch}, pages
  49--101. Wiley, 1962.

\bibitem[EK18]{Kos2}
Benjamin~R. Edwards and V.~Alan Kosteleck\'{y}.
\newblock Riemann-{F}insler geometry and {L}orentz-violating scalar fields.
\newblock {\em Phys. Lett. B}, 786:319--326, 2018.

\bibitem[FP16]{fusterpabst}
Andrea {Fuster} and Cornelia {Pabst}.
\newblock Finsler pp-waves.
\newblock {\em Phys. Rev. D}, 94(10):104072, 2016.

\bibitem[FS06]{flores}
Jos{\'e}~L. Flores and Miguel S{\'a}nchez.
\newblock On the geometry of pp-wave type spacetimes.
\newblock In {\em Analytical and Numerical Approaches to Mathematical
  Relativity}, pages 79--98. Springer, 2006.

\bibitem[FS20]{FS-Ehlers}
Jos{\'e}~L. Flores and Miguel S{\'a}nchez.
\newblock The {E}hlers-{K}undt conjecture about gravitational waves and
  dynamical systems.
\newblock {\em Journal of Differential Equations}, 268(12):7505--7534, 2020.

\bibitem[Ger69]{gerochL}
Robert Geroch.
\newblock Limits of spacetimes.
\newblock {\em Communications in Mathematical Physics}, 13(3):180--193, 1969.

\bibitem[GL16]{globke}
Wolfgang Globke and Thomas Leistner.
\newblock Locally homogeneous pp-waves.
\newblock {\em Journal of Geometry and Physics}, 108:83--101, 2016.

\bibitem[Har13]{harte}
Abraham~I. Harte.
\newblock {Strong lensing, plane gravitational waves and transient flashes}.
\newblock {\em Class. Quant. Grav.}, 30:075011, 2013.

\bibitem[HPF21]{HPF21}
Sjors Heefer, Christian Pfeifer, and Andrea Fuster.
\newblock Randers pp-waves.
\newblock {\em Phys. Rev. D}, 104(2):Paper No. 024007, 10, 2021.

\bibitem[HPV20]{HPV20}
Manuel Hohmann, Christian Pfeifer, and Nicoleta Voicu.
\newblock Relativistic kinetic gases as direct sources of gravity.
\newblock {\em Phys. Rev. D}, 101(2):024062, 13, 2020.

\bibitem[Jav19]{Jav19}
Miguel~\'Angel Javaloyes.
\newblock Anisotropic tensor calculus.
\newblock {\em Int. J. Geom. Methods Mod. Phys.}, 16(suppl. 2):1941001, 26,
  2019.

\bibitem[Jav20]{Jav20}
Miguel~{\'A}ngel Javaloyes.
\newblock Curvature computations in {F}insler geometry using a distinguished
  class of anisotropic connections.
\newblock {\em Mediterranean Journal of Mathematics}, 17(4):1--21, 2020.

\bibitem[JS14]{JS14}
Miguel~\'Angel Javaloyes and Miguel S\'{a}nchez.
\newblock Finsler metrics and relativistic spacetimes.
\newblock {\em Int. J. Geom. Methods Mod. Phys.}, 11(9):1460032, 15, 2014.

\bibitem[JS15]{JavSoa15}
Miguel~\'Angel Javaloyes and Bruno~Learth Soares.
\newblock Geodesics and {J}acobi fields of pseudo-{F}insler manifolds.
\newblock {\em Publ. Math. Debrecen}, 87(1-2):57--78, 2015.

\bibitem[JS20]{JS20}
Miguel~\'Angel Javaloyes and Miguel S\'{a}nchez.
\newblock On the definition and examples of cones and {F}insler spacetimes.
\newblock {\em Rev. R. Acad. Cienc. Exactas F\'{\i}s. Nat. Ser. A Mat. RACSAM},
  114(1):Paper No. 30, 46, 2020.

\bibitem[JSVn21]{JSV22}
Miguel~\'Angel Javaloyes, Miguel S\'anchez, and Fidel~F. Villase\~nor.
\newblock Anisotropic connections and parallel transport in finsler spacetimes.
\newblock 2021.

\bibitem[KRT12]{Kos1}
V.~Alan Kosteleck\'{y}, N.~Russell, and R.~Tso.
\newblock Bipartite {R}iemann-{F}insler geometry and {L}orentz violation.
\newblock {\em Phys. Lett. B}, 716(3-5):470--474, 2012.

\bibitem[Lee18]{Lee}
John~M. Lee.
\newblock {\em Introduction to Riemannian Manifolds}, volume 176.
\newblock Springer, $2^{\text{nd}}$ edition, 2018.

\bibitem[LP18]{Perlick}
Claus L\"{a}mmerzahl and Volker Perlick.
\newblock Finsler geometry as a model for relativistic gravity.
\newblock {\em Int. J. Geom. Methods Mod. Phys.}, 15(suppl. 1):1850166, 20,
  2018.

\bibitem[LS16]{leistner}
Thomas Leistner and Daniel Schliebner.
\newblock Completeness of compact {L}orentzian manifolds with abelian holonomy.
\newblock {\em Mathematische Annalen}, 364(3-4):1469--1503, 2016.

\bibitem[Min17]{Minguzzi}
Ettore Minguzzi.
\newblock Affine sphere relativity.
\newblock {\em Comm. Math. Phys.}, 350(2):749--801, 2017.

\bibitem[O'N83]{o1983}
Barrett O'Neill.
\newblock {\em Semi-{R}iemannian {G}eometry with {A}pplications to
  {R}elativity}, volume 103.
\newblock Academic press, 1983.

\bibitem[Pen65]{remarkable}
Roger Penrose.
\newblock A remarkable property of plane waves in general relativity.
\newblock {\em Reviews of Modern Physics}, 37(1):215, 1965.

\bibitem[Pen72]{pen}
Roger Penrose.
\newblock {\em Techniques of {D}ifferential {T}opology in {R}elativity}.
\newblock SIAM, 1972.

\bibitem[Pen76]{penPW}
Roger Penrose.
\newblock Any space-time has a plane wave as a limit.
\newblock In {\em Differential {G}eometry and {R}elativity}, pages 271--275.
  Springer, 1976.

\bibitem[PW12]{PfeiWol}
Christian Pfeifer and Mattias N.~R. Wohlfarth.
\newblock Beyond the speed of light on {F}insler spacetimes.
\newblock {\em Phys. Lett. B}, 712(3):284--288, 2012.

\bibitem[SHN{\etalchar{+}}17]{AMS}
Christina Sormani, Denson~C. Hill, Pave{\l} Nurowski, Lydia Bieri, David
  Garfinkle, and Nicol{\'a}s Yunes.
\newblock The {M}athematics of {G}ravitational {W}aves: {A} {T}wo-{P}art
  {F}eature.
\newblock {\em Notices of the AMS}, 64(7):684--707, 2017.

\end{thebibliography}
\end{document}